\documentclass[reqno, 11pt]{amsart}
\usepackage{dsfont}
\usepackage{amssymb}
\usepackage{mathrsfs}
\usepackage{enumitem}
\usepackage{comment}
\usepackage{microtype}
\usepackage[pdfstartview=FitH, pdfborder={0 0 0}, colorlinks=true, citecolor=blue, linkcolor=blue]{hyperref}
\usepackage{aliascnt}
\usepackage[backend=biber, style=alphabetic, url=false]{biblatex}
\theoremstyle{plain}
\newtheorem{theorem}{Theorem}[section]

\newaliascnt{conjecture}{theorem}

\aliascntresetthe{conjecture}

\newaliascnt{corollary}{theorem}

\aliascntresetthe{corollary}

\newaliascnt{lemma}{theorem}
\newtheorem{lemma}[lemma]{Lemma}
\aliascntresetthe{lemma}

\newaliascnt{fact}{theorem}

\aliascntresetthe{fact}

\newaliascnt{claim}{theorem}

\aliascntresetthe{claim}

\newaliascnt{proposition}{theorem}
\newtheorem{proposition}[proposition]{Proposition}
\aliascntresetthe{proposition}

\theoremstyle{definition}

\newaliascnt{definition}{theorem}

\aliascntresetthe{definition}

\newaliascnt{example}{theorem}

\aliascntresetthe{example}

\theoremstyle{remark}

\newaliascnt{remark}{theorem}
\newtheorem{remark}[remark]{Remark}
\aliascntresetthe{remark}

\numberwithin{equation}{section}
\newcommand{\sca}{\mathrm{Scal}}
\newcommand{\Ric}{\mathrm{Ric}}

\begin{document}
\title[Nonsmoothable collapsing limits]{Collapse to nonsmoothable manifolds under lower sectional and positive Ricci curvature bounds}
\author[J.~Ge]{Jian Ge}
\address[Ge]{School of Mathematical Sciences, Laboratory of Mathematics and Complex Systems, Beijing Normal University, Beijing 100875, P. R. China.}
\email{jge@bnu.edu.cn}
\thanks{JG is partially supported by the NSFC (No. 12371049) and the Fundamental Research Funds for the Central Universities.}

\author[C.~Qian]{Chao Qian}
\address[Qian]{School of Mathematics and Statistics, Beijing Institute of Technology, Beijing
100081, P.R. China}
\email{6120150035@bit.edu.cn}
\thanks{CQ is partially supported by the NSFC (No. 12371048).}


\begin{abstract}
	We construct smooth collapsing sequences with sectional curvature at least $-1$ whose limits are topological manifolds with nonzero Kirby--Siebenmann class. For every $n\geq6$, the source and limit are simply connected, have dimensions $n+3$ and $n$, and the metrics have a common positive Ricci lower bound. The limits are homeomorphic to the product of the Chern four-manifold with $S^{n-4}$. We also obtain spin examples and five-dimensional lens-space homotopy types, the latter from finite quotients of one fixed $2$-connected eight-manifold. Finally, we realize aspherical manifolds with no closed PL representatives of their homotopy types as collapsing limits with scalar curvatures tending uniformly to $+\infty$, and establish an obstruction to their approximation with almost nonnegative Ricci curvature.
\end{abstract}
\maketitle

\section{Introduction}\label{sec:introduction}

Let $\{(M_i^n,g_i)\}$ be a sequence of closed Riemannian manifolds with $\sec(g_i)\geq-1$, converging in the Gromov--Hausdorff sense to a compact Alexandrov space $X^d$. When $d=n$, Perelman's stability theorem implies that $M_i^n$ and $X^d$ are eventually homeomorphic; see \cite[Theorem 1.1]{KapovitchStability2007}. Under collapse, the topology of the limit can change. Our main result constructs fixed simply connected smooth sources $M^{n+3}$ that collapse to simply connected topological manifolds $X^n$ with nonzero Kirby--Siebenmann class, while the metrics satisfy both a sectional curvature lower bound and a common positive Ricci lower bound. We identify the homeomorphism types of these limits, and extend the construction to spin manifolds and five-dimensional limits with finite cyclic fundamental groups.

For a topological manifold $X^d$ of dimension at least $5$, the Kirby--Siebenmann class
\begin{equation*}
	\operatorname{ks}(X^d)\in H^4(X^d;\mathbb Z/2)
\end{equation*}
vanishes precisely when $X^d$ admits a PL structure; see \cite[Section 3.8]{Quinn}. Its nonvanishing therefore excludes every smooth structure. The regularity of Alexandrov limits also involves the geometry of the limiting metric: Kapovitch constructed Alexandrov metrics on topological spheres that admit no same-dimensional smooth approximation with a common sectional curvature lower bound \cite[Example 1.5]{KapovitchRegularity2002}.

We use $M^d$ for manifolds and $X^d$ for Alexandrov spaces, with superscripts indicating real dimensions; smoothness is specified when needed. We write $\simeq$ for homotopy equivalence and $\cong$ for homeomorphism; diffeomorphisms are stated explicitly. For algebraic objects, $\cong$ denotes isomorphism. We write the acting group on the left or right of a quotient according to the side of its action.

Let $M_{\mathrm{Ch}}^4$ denote the Chern manifold, also written ${}^*\mathbb{CP}^2$: the closed simply connected oriented topological four-manifold with intersection form $\langle1\rangle$ and nonzero Kirby--Siebenmann class. Freedman's classification determines it up to homeomorphism, and $M_{\mathrm{Ch}}^4\simeq\mathbb{CP}^2$; see \cite[Theorems 1.5 and 1.8]{Freedman1982}.

\begin{theorem}\label{thm:simply-connected}
	There exists a constant $c>0$ such that, for every $n\geq6$, there are a closed simply connected smooth manifold $M^{n+3}$ and metrics $g_j$ on $M^{n+3}$ satisfying
	\begin{equation*}
		\sec(g_j)\geq-1,\qquad \Ric(g_j)\geq c\,g_j,
	\end{equation*}
	and $(M^{n+3}, g_j)$ converges in the Gromov--Hausdorff sense to a compact Alexandrov space $X^n$ whose underlying space is a closed simply connected topological manifold with
	\begin{equation*}
		X^n\cong M_{\mathrm{Ch}}^4\times S^{n-4},\qquad \operatorname{ks}(X^n)\ne0.
	\end{equation*}
	In particular, $X^n$ admits no PL structure and no smooth structure.
\end{theorem}

The lower dimension bound is sharp among simply connected closed topological manifolds of dimension at least five. Indeed, for such a manifold $M^5$, Poincar\'e duality gives $H^4(M^5;\mathbb Z/2)\cong H_1(M^5;\mathbb Z/2)=0$, so $\operatorname{ks}(M^5)=0$ and $M^5$ admits a PL structure.

A sphere-bundle variant gives simply connected spin examples on both sides of the collapse; see \autoref{prop:spin-examples}. Recall that a topological manifold is called spin if it is orientable and the second Stiefel–Whitney class of its stable topological tangent bundle vanishes.

For a positive integer $r$, let $L^5(r;1,1,1)$ denote the quotient of the unit sphere $S^5\subset\mathbb C^3$ by the scalar action generated by multiplication by $e^{2\pi i/r}$. A circle-bundle variant gives $5$-dimensional limits with finite cyclic fundamental groups.

\begin{theorem}\label{thm:lens-family}
	There are a closed $2$-connected smooth manifold $M_0^8$ with a free circle action and a constant $c>0$ such that, for every integer $m\geq1$, the quotient $M_m^8=(\mathbb Z/(2m))\backslash M_0^8$ admits metrics $g_{m,j}$ such that, for a compact Alexandrov space $(X_m^5,d_m)$,
	\begin{equation}\label{eq:uniform-bounds}
		\sec(g_{m,j})\geq-1,\qquad \Ric(g_{m, j})\geq c \,g_{m, j},
	\end{equation}
	and
	\begin{equation*}
		(M_m^8, g_{m,j})\xrightarrow[j\to\infty]{\mathrm{GH}}(X_m^5,d_m).
	\end{equation*}
	The underlying space $X_m^5$ is a closed orientable topological manifold, and
	\begin{equation}\label{eq:lens-topology}
		X_m^5\simeq L^5(2m;1,1,1),\qquad \operatorname{ks}(X_m^5)\ne0.
	\end{equation}
	Moreover,
	\begin{equation*}
		\pi_1(M_m^8)\cong\pi_1(X_m^5)\cong\mathbb Z/(2m),
	\end{equation*}
	and $M_0^8$ is the universal covering manifold of every $M_m^8$. Consequently the spaces $X_m^5$ admit no PL structures and are pairwise nonhomeomorphic. The constant $c$ is independent of both $m$ and $j$.
\end{theorem}

For $m=1$, the limit is homotopy equivalent to $\mathbb{RP}^5$ but has no smooth structure. Such topological manifolds already occur in the classification of $5$-manifolds with fundamental group of order two \cite[discussion preceding Theorem 3.11]{HS}. \autoref{thm:lens-family} realizes all the even-order lens-space homotopy types in \eqref{eq:lens-topology} by Alexandrov metrics and smooth collapsing sequences satisfying the common bounds \eqref{eq:uniform-bounds}. The sources share one fixed $2$-connected universal covering manifold, and the bounds are uniform over the entire family.

Finally, we construct aspherical limits. Davis and Januszkiewicz produced closed aspherical manifolds whose homotopy types cannot be represented by PL manifolds \cite[Theorem 5a.1 and Corollary 5a.4]{DJ}. We realize their torus-product examples as Alexandrov spaces and as smooth collapsing limits with collapse codimension $3$. The scalar curvatures of the smooth sources tend uniformly to $+\infty$. An application of fundamental group structure theorems also obstructs approximation of these limits with almost nonnegative Ricci curvature in any fixed dimension.

\begin{theorem}\label{thm:aspherical}
	For every $n\geq5$, there are a closed smooth manifold $M^{n+3}$, a constant $D_n>0$, and metrics $h_j$ on $M^{n+3}$ satisfying
	\begin{equation*}
		\begin{gathered}
		\sec(h_j)\geq-1,\qquad \operatorname{diam}(M^{n+3},h_j)\leq D_n,\\
		\inf_{M^{n+3}}\sca(h_j)\longrightarrow+\infty,
		\end{gathered}
	\end{equation*}
	such that $(M^{n+3}, h_j)$ converges in the Gromov--Hausdorff sense to a compact Alexandrov space $X^n$. Its underlying space is a closed orientable aspherical topological manifold with $\operatorname{ks}(X^n)\ne0$, and is not homotopy equivalent to any closed PL $n$-manifold. The same underlying space also admits a locally $\operatorname{CAT}(0)$ metric. These spaces cannot be Gromov--Hausdorff limits of closed Riemannian manifolds of any fixed dimension $N\geq2$ with Ricci lower bounds tending to zero from below.
\end{theorem}

Galaz-Garc\'ia and Zarei established the existence of Alexandrov metrics on closed topological manifolds of cohomogeneity one \cite[Corollary B]{GalazGarciaZarei2018}. Michael Kapovich constructed compact three-dimensional aspherical Alexandrov spaces with isolated cone singularities over $\mathbb{RP}^2$ \cite[Theorem 1]{KapovichAspherical2025}. Alattar and Tadman characterized approximation of resolvable ANR homology manifolds by PL manifolds of the same dimension with a uniform contractibility function, using an extension of the Kirby--Siebenmann class \cite[Corollary 4]{AlattarTadman2026}. The constructions here realize the specified non-PL topological manifolds by Alexandrov metrics and smooth collapse of codimension three, with the additional curvature bounds stated above.

The common local ingredient is the quotient $I^*\backslash\mathbb R^4$, where $I^*\subset SU(2)$ is the binary icosahedral group. Its link is the Poincar\'e homology sphere $H^3=I^*\backslash S^3$, whose Rokhlin invariant is nonzero. The double suspension theorem makes the product of $I^*\backslash\mathbb R^4$ with a positive-dimensional manifold locally Euclidean. The link formula then identifies the Kirby--Siebenmann class with the dual of the codimension-four singular stratum. A diagonal $I^*$-action on $D^4\times SU(2)$ gives a smooth local total space. An explicit boundary trivialization glues it to a trivial bundle over the regular part, producing a smooth total space of dimension three greater than the quotient.

For the positive Ricci construction, we use a classical degree-one del Pezzo surface $Q_+^4$ with a single $E_8$ singularity. Its minimal resolution is $M_{\mathrm{res}}^4=\mathbb{CP}^2\#8\,\overline{\mathbb{CP}}^{\,2}$, and contracting the exceptional $E_8$ configuration amounts to capping the complement of the negative $E_8$ plumbing by $C(H^3)$. The surface $Q_+^4$ is homotopy equivalent to $\mathbb{CP}^2$; see \autoref{lem:fano-topology}. The orbifold Calabi--Yau theorem \cite[Theorem 1.2]{Faulk} supplies a K\"ahler metric on $Q_+^4$ with prescribed positive Ricci form. The smooth total-space construction in \autoref{lem:smooth-total-space} gives a closed simply connected smooth manifold $M_+^7$ with a locally free $SU(2)$-action and orbifold quotient $Q_+^4$. Shrinking its $SU(2)$-orbits and taking products with $S^{n-4}$ proves \autoref{thm:simply-connected}. Circle bundles with even Euler class give the five-dimensional manifolds in \autoref{thm:lens-family}. We then collapse the $SU(2)$ fibres of their smooth total spaces. The orbit-shrinking construction and positive Ricci lifting are standard tools; see \cite[Remark 3.3 and Theorem A]{SearleWilhelm2015}. We record the connection-metric estimates that give both curvature bounds along the same collapsing family. Finite circle quotients of one fixed construction make those estimates uniform in $m$. Twisting the sphere factor by the Cartier line bundle on $Q_+^4$ gives the spin examples in \autoref{prop:spin-examples}.

For \autoref{thm:aspherical}, we use the hyperbolized capped $E_8$ plumbing of Davis--Januszkiewicz \cite[Lemma 5a.3]{DJ}. We remove a cone neighborhood, smooth its PL complement, and attach the standard orbifold cone. This gives a smooth orbifold to which the $SU(2)$ construction applies. This realizes the Davis--Januszkiewicz torus-product homotopy types by Alexandrov metrics and smooth collapse with the scalar-curvature growth in \autoref{thm:aspherical}.

The proof is organized as follows. \autoref{sec:construction} establishes the smooth total-space construction, the curvature estimates, and the link formula used to detect the obstruction. \autoref{sec:positive} proves \autoref{thm:simply-connected} and \autoref{thm:lens-family}, and constructs the spin variant. \autoref{sec:aspherical} proves \autoref{thm:aspherical} and the Ricci obstruction.

\section{Smooth total spaces and collapsing metrics}\label{sec:construction}

All orbifolds considered below are effective and have no boundary unless otherwise stated. We distinguish an orbifold $\mathcal X^d$ from its underlying space $X^d$ when this distinction is needed. A smooth metric on a compact orbifold has a finite sectional curvature lower bound in its uniformizing charts, and the induced length metric on $X^d$ is an Alexandrov metric with the same lower bound. We normalize the bi-invariant metric $g_G$ on $G=SU(2)$ to have sectional curvature $1$, so that $\operatorname{diam}(G,g_G)=\pi$ and $\Ric(g_G)=2g_G$. The $I^*$-actions and scalar circle actions are on the left; the $G$-actions on smooth total spaces are on the right.

We first construct smooth seven-dimensional total spaces over orbifolds with a single Poincar\'e-link singularity.

\begin{lemma}\label{lem:smooth-total-space}
	Let $\mathcal X^4$ be a closed smooth four-orbifold with exactly one singular point, whose local model is $I^*\backslash D^4$ with the standard action $I^*\subset SU(2)$ on $D^4\subset \mathbb C^2$. Then there is a closed smooth manifold $M^7$ with a locally free smooth $G$-action and orbifold quotient $M^7/G=\mathcal X^4$. If the complement $W^4$ of an open singular ball is simply connected, then $M^7$ is simply connected.
\end{lemma}

\begin{proof}
	Identify $\partial W^4$ with $H^3=I^*\backslash S^3$. Let
	\begin{equation*}
		M_{\mathrm{loc}}^7=I^*\backslash(D^4\times G),\qquad
		\gamma\cdot(x,a)=(\gamma x,\gamma a).
	\end{equation*}
	The diagonal finite-group action is free, including over $x=0$, because left multiplication on $G$ is free. Thus $M_{\mathrm{loc}}^7$ is a compact smooth manifold with boundary. Right multiplication on the second factor descends to a locally free $G$-action, with quotient $I^*\backslash D^4$.

	The boundary bundle has the explicit $G$-equivariant trivialization
	\begin{equation}\label{eq:boundary-trivialization}
		\Phi:I^*\backslash(S^3\times G)\longrightarrow H^3\times G,
		\qquad [x,a]\longmapsto([x],x^{-1}a).
	\end{equation}
	It is well defined because $(\gamma x)^{-1}(\gamma a)=x^{-1}a$. Its inverse sends $([x],a)$ to $[x,xa]$, independently of the chosen lift $x$. Using the boundary identification and smooth collars, form
	\begin{equation}\label{eq:smooth-total-space}
		M^7=(W^4\times G)\cup_{\Phi}M_{\mathrm{loc}}^7.
	\end{equation}
	The two actions glue, and their quotient is the original orbifold $\mathcal X^4$.

	For the fundamental group assertion, $M_{\mathrm{loc}}^7$ deformation retracts to $I^*\backslash G$, and $\partial M_{\mathrm{loc}}^7\to M_{\mathrm{loc}}^7$ induces an isomorphism on fundamental groups. This can be checked upstairs from the simply connected spaces $S^3\times G$ and $D^4\times G$, whose covering group is $I^*$. If $W^4$ is simply connected, then so is $W^4\times G$, and van Kampen's theorem applied to \eqref{eq:smooth-total-space} gives $\pi_1(M^7)=0$.
\end{proof}

The following proposition collects standard connection-metric estimates for the seven-dimensional construction and its circle- and sphere-bundle pullbacks. Orbit shrinking with a common sectional curvature lower bound and positive Ricci lifting are described in \cite[Remark 3.3 and Theorem A]{SearleWilhelm2015}. We include the estimates needed to control both curvature bounds along a single collapsing family.

\begin{proposition}\label{prop:collapse-metrics}
	Let $M^{d+3}$ be a closed smooth manifold with a locally free $G$-action, let $\pi:M^{d+3}\to\mathcal X^d=M^{d+3}/G$ be the quotient map, and let $q$ be a smooth orbifold metric on $\mathcal X^d$. There are $G$-invariant metrics $g_t$ on $M^{d+3}$, $0<t\leq t_0$, and constants $K,D<\infty$ such that
	\begin{equation}\label{eq:collapse-estimates}
		\begin{gathered}
		\sec(g_t)\geq-K,\qquad \operatorname{diam}(M^{d+3},g_t)\leq D,\\
		d_{\mathrm{GH}}((M^{d+3},g_t),(X^d,q))\leq\frac{\pi t}{2}.
		\end{gathered}
	\end{equation}
	The scalar curvatures satisfy $\inf_{M^{d+3}}\sca(g_t)\to+\infty$ as $t\to0$. If $q$ has positive Ricci curvature, $t_0$ can be chosen so that $\Ric(g_t)\geq c_0g_t$ for a constant $c_0>0$. If a compact group acts smoothly on $M^{d+3}$, commutes with $G$, and induces isometries of $(\mathcal X^d,q)$, then all the metrics can be chosen invariant under that group as well.
\end{proposition}

\begin{proof}
	Choose a $G$-equivariant connection form $A$ and, for $0<t\leq1$, set
	\begin{equation}\label{eq:connection-metrics}
		g_t=\pi^*q+t^2g_G(A,A).
	\end{equation}
	Local freeness ensures that these are smooth metrics on $M^{d+3}$. If a commuting compact group is given, average $A$ over that group; the resulting metrics have the required additional invariance.

	For $t<1$, the metric $g_t$ is the Cheeger deformation of $g_1$ with parameter $t^{-2}-1$ \cite{Cheeger}: it is induced by the free diagonal quotient of
	\begin{equation*}
		(M^{d+3}\times G,\;g_1\oplus(t^{-2}-1)^{-1}g_G).
	\end{equation*}
	Since $g_G$ has nonnegative sectional curvature, O'Neill's formula \cite{ONeill} gives
	\begin{equation*}
		\sec(g_t)\geq\min\{\inf\sec(g_1),0\}.
	\end{equation*}
	Also $g_t\leq g_1$, giving the diameter bound. The quotient metric is $q$, and each orbit has intrinsic diameter at most $\pi t$. The quotient-distance formula therefore gives
	\begin{equation*}
		0\leq d_{g_t}(z,z')-d_q(\pi z,\pi z')\leq\pi t.
	\end{equation*}
	The graph of $\pi$ is a correspondence of distortion at most $\pi t$, proving the Gromov--Hausdorff bound.

	Tracing the connection-metric curvature formulas in \cite[Proposition 3.6]{HST} gives, in a $g_t$-orthonormal horizontal--vertical splitting,
	\begin{equation}\label{eq:ricci-block}
		\Ric(g_t)=
		\begin{pmatrix}
		\Ric(q)+O(t^2)&O(t)\\
		O(t)&2t^{-2}I_3+O(t^2)
		\end{pmatrix}.
	\end{equation}
	Here the horizontal block denotes the pullback of $\Ric(q)$. The calculation applies in finite orbifold charts, and compactness makes the error estimates uniform in operator norm. Taking the trace yields
	\begin{equation}\label{eq:scalar-collapse}
		\sca(g_t)=6t^{-2}+\sca(q)\circ\pi+O(t^2),
	\end{equation}
	which proves uniform divergence of the scalar curvature.

	If $\Ric(q)\geq\lambda q$ for $\lambda>0$, the principal diagonal part of \eqref{eq:ricci-block} is bounded below by $\min\{\lambda,2\}I$ for $t\leq1$, while the remaining matrix has operator norm $O(t)$. After restricting to $0<t\leq t_0$, this gives $\Ric(g_t)\geq c_0g_t$ with $c_0=\min\{\lambda/2,1\}>0$.
\end{proof}

We next record the topological facts used to recognize the quotient manifolds and compute their obstruction. The double suspension theorem of Edwards and Cannon \cite{EdwardsSuspensions2006,CannonDoubleSuspension1979} implies
\begin{equation}\label{eq:local-double-suspension}
	C^\circ(H^3)\times\mathbb R\cong\mathbb R^5.
\end{equation}
Here $C^\circ(H^3)$ denotes the open cone on $H^3$. In particular, $(I^*\backslash D^4)\times\mathbb R^k$ is locally Euclidean at the singular stratum whenever $k\geq1$.

\begin{lemma}\label{lem:link-obstruction}
	Let $\Sigma^{n-4}$ be a closed oriented smooth manifold, with $n\geq5$, and let $M_0^n$ be a compact oriented smooth manifold with boundary diffeomorphic to $H^3\times\Sigma^{n-4}$. Form the closed oriented topological manifold
	\begin{equation*}
		X^n=M_0^n\cup_{H^3\times\Sigma^{n-4}}\bigl(C(H^3)\times\Sigma^{n-4}\bigr)
	\end{equation*}
	using an orientation-compatible boundary diffeomorphism. Let $v$ be the cone vertex and let $i:\Sigma^{n-4}\to X^n$ be given by $i(s)=(v,s)$. Then
	\begin{equation}\label{eq:ks-link}
		\operatorname{ks}(X^n)\cap[X^n]=i_*[\Sigma^{n-4}]\in H_{n-4}(X^n;\mathbb Z/2).
	\end{equation}
\end{lemma}

\begin{proof}
	Choose a PL homology-manifold model compatible with the smooth boundary gluing and the product cone. The smooth complement contributes zero to its local obstruction cycle for the Kirby--Siebenmann class. Along $\{v\}\times\Sigma^{n-4}$, the normal link is $H^3$, and its coefficient in this cycle is the Rokhlin invariant $\mu(H^3)$, by the Cohen--Saito--Sullivan link formula; see \cite[Proposition 3.7 and Section 3.8]{Quinn}. The Poincar\'e sphere bounds the smooth spin $E_8$ plumbing of signature $\pm8$, so $\mu(H^3)=1$ in $\mathbb Z/2$. Thus the obstruction cycle is the mod-$2$ fundamental cycle of $\{v\}\times\Sigma^{n-4}$, proving \eqref{eq:ks-link}.
\end{proof}

\section{Simply connected limits and lens-space examples}\label{sec:positive}

Consider the classical degree-one del Pezzo surface
\begin{equation*}
	Q_+^4=\{w^2+z^3+xy^5=0\}\subset\mathbb P(1,1,2,3),
\end{equation*}
where $x,y,z,w$ have weights $1,1,2,3$, respectively; see \cite[Table 3, no.~$1^\circ$]{CheltsovProkhorov2021}. It has a single $E_8$ singularity at $v=[1:0:0:0]$, with local chart $I^*\backslash\mathbb C^2$. Write $\mathcal Q_+^4$ for its natural effective complex orbifold structure. Its anticanonical line bundle $L=-K_{Q_+^4}=\mathcal O_{Q_+^4}(1)$ is ample and Cartier. With $u=c_1(L)$, the degree-one condition gives $\langle u^2,[Q_+^4]\rangle=1$.

The minimal resolution is a weak del Pezzo surface of degree one, whose underlying smooth manifold is $M_{\mathrm{res}}^4=\mathbb{CP}^2\#8\,\overline{\mathbb{CP}}^{\,2}$; its exceptional curves form the $E_8$ configuration of eight $(-2)$-spheres \cite[Section 2 and Appendix A, no.~$1^\circ$]{CheltsovProkhorov2021}. Let $W_E^4$ be a regular neighborhood of these curves, the negative $E_8$ plumbing with boundary $H^3$, and let $W_+^4$ be its complement. Thus
\begin{equation}\label{eq:plumbing-decomposition}
	M_{\mathrm{res}}^4=W_E^4\cup_{H^3}W_+^4,
	\qquad Q_+^4=W_+^4\cup_{H^3}C(H^3).
\end{equation}
Here $\partial W_+^4=-H^3$, and the cone in the second decomposition is oriented accordingly. We record the elementary topological consequences needed below.

\begin{lemma}\label{lem:fano-topology}
	The manifold $W_+^4$ is homotopy equivalent to $S^2$. There is a homotopy equivalence $f:Q_+^4\to\mathbb{CP}^2$ such that $f^*h=u$, where $h=c_1(\mathcal O_{\mathbb{CP}^2}(1))$.
\end{lemma}

\begin{proof}
	The open set $U=\{y\ne0\}\subset Q_+^4$ is isomorphic to $\mathbb C^2$, since setting $y=1$ determines $x=-w^2-z^3$. Its complement in $Q_+^4\setminus\{v\}$ is a smooth complex curve, so general position moves every loop in the smooth locus into $U$. Thus the smooth locus, and hence its deformation retract $W_+^4$, is simply connected. The cone decomposition in \eqref{eq:plumbing-decomposition} also gives $\pi_1(Q_+^4)=0$ by van Kampen's theorem.

	Since $H^3$ is an integral homology sphere, Mayer--Vietoris for the two decompositions in \eqref{eq:plumbing-decomposition} gives $\widetilde H_i(W_+^4;\mathbb Z)=\mathbb Z$ for $i=2$ and zero otherwise, and
	\begin{equation*}
		H_i(Q_+^4;\mathbb Z)\cong H_i(\mathbb{CP}^2;\mathbb Z).
	\end{equation*}
	A map $S^2\to W_+^4$ representing a generator of $H_2(W_+^4;\mathbb Z)$ is a homotopy equivalence by the Hurewicz and homological Whitehead theorems.

	The space $Q_+^4$ has a finite four-dimensional CW structure, so a classifying map for $u$ into $\mathbb{CP}^{\infty}$ can be cellularly approximated into $\mathbb{CP}^2$. Since $\langle u^2,[Q_+^4]\rangle=1$, the resulting map $f$ induces an isomorphism on integral homology. It is therefore a homotopy equivalence by the homological Whitehead theorem.
\end{proof}

\begin{lemma}\label{lem:positive-ricci-base}
	The compact effective orbifold $\mathcal Q_+^4$ admits a K\"ahler metric $q$ and a constant $\lambda>0$ such that $\Ric(q)\geq\lambda q$, with K\"ahler form $\omega$ satisfying $[\omega/(2\pi)]=u$.
\end{lemma}

\begin{proof}
	The anticanonical line bundle $L=-K_{Q_+^4}$ is ample and Cartier. Since $I^*\subset SU(2)$ acts trivially on the local canonical line, the orbifold canonical bundle is the pullback of the Cartier canonical bundle. Hence
	\begin{equation*}
		c_1^{\mathrm{orb}}(\mathcal Q_+^4)=u.
	\end{equation*}
	We verify that $2\pi u$ is an orbifold K\"ahler class. A projective embedding defined by a sufficiently large power $L^N$ pulls back the Fubini--Study form, normalized to represent $2\pi c_1(\mathcal O(1))$. Dividing by $N$ gives a smooth closed orbifold $(1,1)$-form $\eta$ representing $2\pi u$, semipositive everywhere and positive away from $v$.

	On a small $I^*$-invariant uniformizing ball at $v$, choose linear coordinates $\zeta$ and an invariant cutoff $\chi$ supported in that ball and equal to one on a smaller ball. Extend the correction by zero and set
	\begin{equation*}
		\eta_+=\eta+\varepsilon\sqrt{-1}\partial\bar\partial(\chi|\zeta|^2).
	\end{equation*}
	The correction is positive definite on the inner ball, while $\eta$ is already positive on the compact cutoff annulus. Thus, for sufficiently small $\varepsilon>0$, $\eta_+$ is a smooth strictly positive closed orbifold $(1,1)$-form with $[\eta_+]=2\pi u$.

	Apply the orbifold Calabi--Yau theorem \cite[Theorem 1.2]{Faulk}, taking $\eta_+$ as the initial K\"ahler form and prescribing the Ricci form to be $\eta_+$. It yields a smooth orbifold K\"ahler form $\omega$ such that
	\begin{equation*}
		[\omega]=[\eta_+]=2\pi u,\qquad \rho_\omega=\eta_+>0.
	\end{equation*}
	The associated orbifold Riemannian metric $q$ therefore has positive Ricci curvature, and compactness gives $\Ric(q)\geq\lambda q$ for some $\lambda>0$.
\end{proof}

\begin{proof}[Proof of \autoref{thm:simply-connected}]
	Apply \autoref{lem:smooth-total-space} to $\mathcal Q_+^4$ to obtain a closed smooth manifold $M_+^7$. Apply \autoref{prop:collapse-metrics} directly to $M_+^7\to\mathcal Q_+^4$ with the positive-Ricci metric $q$ of \autoref{lem:positive-ricci-base}. Put $k=n-4\geq2$ and take the product with a unit round $S^k$. The smooth total space is $M^{n+3}=M_+^7\times S^k$, and the limit is $X^n=Q_+^4\times S^k$. If the metrics on $M_+^7$ have Ricci curvature at least $c_7>0$ and sectional curvature at least $-K_7$, then the product has Ricci curvature at least $\min\{c_7,1\}$, since $k-1\geq1$, and sectional curvature at least $-\max\{K_7,0\}$. A fixed rescaling by $\Lambda=\max\{K_7,1\}$ gives the stated bounds with $c=\min\{c_7,1\}/\Lambda$, independently of $n$.

	The double suspension theorem makes $X^n$ a topological manifold. A neighborhood of its singular stratum $\{v\}\times S^k$ is $C(H^3)\times S^k$, and its complement is smooth. By \autoref{lem:link-obstruction},
	\begin{equation*}
		\operatorname{ks}(X^n)\cap[X^n]=[\{v\}\times S^k]\ne0,
	\end{equation*}
	where nonvanishing follows by projection to $S^k$. \autoref{lem:fano-topology} shows that both $W_+^4$ and $Q_+^4$ are simply connected. Hence $M_+^7$ is simply connected by \autoref{lem:smooth-total-space}; since $k\geq2$, both $M^{n+3}$ and $X^n$ are simply connected.

	To identify the homeomorphism type, choose a compact contractible topological manifold $M_{\mathrm{ctr}}^4$ with boundary $H^3$, using \cite[Theorem 1.4']{Freedman1982}, and form
	\begin{equation*}
		M_{\mathrm{cap}}^4=W_+^4\cup_{H^3}M_{\mathrm{ctr}}^4.
	\end{equation*}
	Van Kampen's theorem and the Mayer--Vietoris sequence show that $M_{\mathrm{cap}}^4$ is simply connected and has intersection form $\langle1\rangle$. Retain a boundary collar in $M_{\mathrm{ctr}}^4$ and collapse its remaining compact contractible core to the cone vertex. Extending by the identity on $W_+^4$ gives a cell-like map $M_{\mathrm{cap}}^4\to Q_+^4$. Its product with the identity of $S^k$ is a cell-like map
	\begin{equation*}
		M_{\mathrm{cap}}^4\times S^k\longrightarrow Q_+^4\times S^k=X^n.
	\end{equation*}
	Both spaces are closed topological manifolds of dimension $n\geq6$, so the cell-like approximation theorem \cite[Approximation Theorem A]{SiebenmannApprox1972} gives a homeomorphism. The product formula for the stable tangent Kirby--Siebenmann class gives
	\begin{equation*}
		\operatorname{ks}(M_{\mathrm{cap}}^4\times S^k)
		=\operatorname{pr}_1^*\operatorname{ks}(M_{\mathrm{cap}}^4).
	\end{equation*}
	The left-hand side is nonzero by the calculation for $X^n$. Thus $\operatorname{ks}(M_{\mathrm{cap}}^4)\ne0$, and Freedman's classification identifies $M_{\mathrm{cap}}^4$ with $M_{\mathrm{Ch}}^4$. This proves the asserted homeomorphism type.
\end{proof}

For every integer $r\geq1$, let $p_r:\widehat X_r^5\to Q_+^4$ be the unit-circle bundle of $L^{\otimes r}$, viewed also as a smooth orbifold circle bundle. The line bundle is Cartier on the coarse space, so the local $I^*$-action on its fibre is trivial. In particular, the underlying map is an ordinary topological principal circle bundle, including over $v$. Its local model at the distinguished fibre is
\begin{equation}\label{eq:circle-local-model}
	(I^*\backslash D^4)\times S^1\longrightarrow I^*\backslash D^4.
\end{equation}
We use $\widehat{\mathcal X}_r^5$ when referring specifically to its smooth orbifold structure.

\begin{proposition}\label{prop:lens-topology}
	The space $\widehat X_r^5$ is a closed orientable topological manifold and satisfies
	\begin{equation*}
		\widehat X_r^5\simeq L^5(r;1,1,1),\qquad \pi_1(\widehat X_r^5)=\mathbb Z/r.
	\end{equation*}
	Furthermore, $\operatorname{ks}(\widehat X_r^5)\ne0$ if and only if $r$ is even.
\end{proposition}

\begin{proof}
	The local model \eqref{eq:circle-local-model} and \eqref{eq:local-double-suspension} give the topological manifold property along $\gamma_r=p_r^{-1}(v)$. Away from this circle the total space is smooth. The complex orientation of the base and the orientation of the fibre orient $\widehat X_r^5$. The circle bundle is trivial over the contractible cone neighborhood of $v$, with a smooth trivialization along its boundary. Thus a neighborhood of $\gamma_r$ is $C(H^3)\times S^1$, attached to the smooth complement along $H^3\times S^1$.

	The homotopy equivalence $f$ of \autoref{lem:fano-topology} pulls the circle bundle of Chern class $r$ over $\mathbb{CP}^2$ back to $p_r$. The induced map of total spaces is a homotopy equivalence, since it covers a homotopy equivalence and restricts to an isomorphism on each fibre. The circle bundle of Chern class $r$ over $\mathbb{CP}^2$ has total space $L^5(r;1,1,1)$, up to reversing the circle convention. This proves the homotopy assertion.

	The normal link along $\gamma_r$ is $H^3$. Thus \autoref{lem:link-obstruction} gives
	\begin{equation}\label{eq:ks-circle}
		\operatorname{ks}(\widehat X_r^5)\cap[\widehat X_r^5]=[\gamma_r]\in H_1(\widehat X_r^5;\mathbb Z/2).
	\end{equation}
	The circle-bundle homotopy sequence contains
	\begin{equation*}
		\pi_2(Q_+^4)\cong\mathbb Z\xrightarrow{\times r}\mathbb Z
		\longrightarrow\pi_1(\widehat X_r^5)\longrightarrow0.
	\end{equation*}
	Consequently the fibre represents a generator of $\pi_1(\widehat X_r^5)=\mathbb Z/r$. Its mod-$2$ homology class is nonzero precisely when $r$ is even. Equation \eqref{eq:ks-circle} and Poincar\'e duality prove the asserted equivalence.
\end{proof}

For even $r$, these spaces have the homotopy types of smooth lens spaces but nonzero Kirby--Siebenmann classes. The parity restriction is necessary for any non-PL closed manifold $M^5$ with finite cyclic fundamental group: when the group has odd order, Poincar\'e duality gives
\begin{equation*}
	H^4(M^5;\mathbb Z/2)\cong H_1(M^5;\mathbb Z/2)=0.
\end{equation*}

\begin{lemma}\label{lem:circle-ricci}
	There is a circle-invariant orbifold metric $q_1$ on $\widehat{\mathcal X}_1^5$ with positive Ricci curvature. The finite quotients $\widehat X_r^5=(\mathbb Z/r)\backslash\widehat X_1^5$ carry induced metrics $q_r$ with a common sectional curvature lower bound and a common positive Ricci curvature lower bound.
\end{lemma}

\begin{proof}
	Take the K\"ahler metric from \autoref{lem:positive-ricci-base}. Choose a real connection form $\theta$ on $\widehat X_1^5$, with fibre period $2\pi$, so that $d\theta=p_1^*\omega$. Such a connection exists because $[\omega/(2\pi)]=c_1(L)$. For $a>0$, define
	\begin{equation*}
		q_1=p_1^*q+a^2\theta^2.
	\end{equation*}
	For a horizontal unit vector $\xi$ and the vertical unit vector $U$, the connection-metric Ricci formulas \cite{ONeill}, together with \cite[Proposition 3.6]{HST}, give
	\begin{align*}
		\Ric_{q_1}(\xi,\xi)&=\Ric_q(\xi,\xi)-\frac{a^2}{2}|\iota_\xi\omega|^2,\\
		\Ric_{q_1}(U,U)&=\frac{a^2}{2}|\omega|^2,\\
		\Ric_{q_1}(\xi,U)&=0.
	\end{align*}
	The last identity follows from $\delta\omega=0$, since the K\"ahler form is parallel. For a two-form $F$, use the norm $|F|^2=\sum_{i<j}F_{ij}^2$. Then $|\omega|^2=2$ and $|\iota_\xi\omega|^2=1$. If $\Ric(q)\geq\lambda q$, choose $0<a^2\leq\lambda$. Then
	\begin{equation*}
		\Ric(q_1)\geq\min\{\lambda/2,a^2\}\,q_1>0.
	\end{equation*}
	The radius $a$ is fixed from now on.

	The map $L\to L^{\otimes r}$, $\xi\mapsto\xi^{\otimes r}$, identifies $\widehat X_r^5$ with the quotient of $\widehat X_1^5$ by its free subgroup $\mathbb Z/r\subset S^1$. The metric $q_1$ is invariant under the circle. These finite quotients are local isometries in orbifold charts, so their sectional and Ricci bounds are the same as those of $q_1$.
\end{proof}

We now assemble the smooth collapsing sequences. Use the manifold $M_+^7$ and quotient map $M_+^7\to\mathcal Q_+^4$ constructed in the proof of \autoref{thm:simply-connected}. Pull the circle bundle $p_1:\widehat X_1^5\to Q_+^4$ back to $M_+^7$ and denote the resulting total space by $M_0^8$. This is a closed smooth manifold. Over the distinguished orbifold chart its model is
\begin{equation*}
	I^*\backslash(D^4\times G\times S^1),
\end{equation*}
where $I^*$ acts diagonally on the first two factors and trivially on the last. Thus $M_0^8$ carries a locally free $G$-action and a free commuting circle action, with
\begin{equation}\label{eq:commuting-quotients}
	M_0^8/G=\widehat{\mathcal X}_1^5,\qquad S^1\backslash M_0^8=M_+^7.
\end{equation}

\begin{lemma}\label{lem:common-cover}
	The manifold $M_0^8$ is $2$-connected, and its quotient by any finite subgroup of the circle has $M_0^8$ as its universal covering manifold.
\end{lemma}

\begin{proof}
	Let $\pi:M_+^7\to Q_+^4$ be the quotient map. In the decomposition \eqref{eq:smooth-total-space}, the intersection is $H^3\times S^3$, whose first and second integral homology vanish, and $M_{\mathrm{loc}}^7\simeq H^3$. The Mayer--Vietoris sequence therefore gives
	\begin{equation*}
		H_2(M_+^7;\mathbb Z)\cong H_2(W_+^4;\mathbb Z)\cong\mathbb Z.
	\end{equation*}
	Under this identification, $\pi^*u$ pairs to one with a generator, since $u|_{W_+^4}$ does. The manifold $M_+^7$ is simply connected, so the Hurewicz theorem identifies its second homotopy group with this homology group. The Euler class of $M_0^8\to M_+^7$ is $\pi^*u$, and the homotopy sequence contains
	\begin{equation*}
		0\longrightarrow\pi_2(M_0^8)\longrightarrow\mathbb Z
		\xrightarrow{\cong}\mathbb Z\longrightarrow\pi_1(M_0^8)\longrightarrow0.
	\end{equation*}
	Hence $\pi_1(M_0^8)=\pi_2(M_0^8)=0$. In particular, for every $m\geq1$ the quotient map $M_0^8\to (\mathbb Z/(2m))\backslash M_0^8$ is the universal covering map.
\end{proof}

\begin{proof}[Proof of \autoref{thm:lens-family}]
	Apply \autoref{prop:collapse-metrics} to $M_0^8\to\widehat{\mathcal X}_1^5$, using the positive-Ricci metric $q_1$ of \autoref{lem:circle-ricci}. Average the connection over the commuting circle, so that every $g_t$ is circle invariant. For each $m\geq1$, let
	\begin{equation*}
		M_m^8=(\mathbb Z/(2m))\backslash M_0^8,\qquad X_m^5=\widehat X_{2m}^5.
	\end{equation*}
	The finite circle action on $M_0^8$ is free, hence $M_m^8$ is a closed smooth manifold. Its induced metrics $g_{m,t}$ have the same local sectional and Ricci bounds as $g_t$. The two actions in \eqref{eq:commuting-quotients} commute, so the remaining $G$-quotient is $(X_m^5,q_{2m})$. Its fibres still have intrinsic diameter at most $\pi t$, giving
	\begin{equation*}
		d_{\mathrm{GH}}((M_m^8,g_{m,t}),(X_m^5,q_{2m}))\leq\frac{\pi t}{2}
	\end{equation*}
	uniformly in $m$.

	Choose $\Lambda\geq\max\{K,1\}$, with $K$ as in \eqref{eq:collapse-estimates}, and multiply every metric by $\Lambda$. Sectional curvature is then at least $-1$, and the positive Ricci lower bound becomes $c=c_0/\Lambda>0$. Choose any sequence $t_j\downarrow0$ in the permitted interval and set $g_{m,j}=\Lambda g_{m,t_j}$ and $d_m=\sqrt\Lambda\,d_{q_{2m}}$. The constant $c$ is independent of $m,j$, and the convergence follows from the preceding estimate. The topological assertions follow from \autoref{prop:lens-topology} and \autoref{lem:common-cover}. Since $\pi_1(X_m^5)=\mathbb Z/(2m)$, different values of $m$ give nonhomeomorphic spaces.
\end{proof}

The product construction has a spin analogue obtained by replacing the sphere factor with a sphere bundle.

\begin{proposition}\label{prop:spin-examples}
	For every $n\geq6$, there are a closed simply connected spin smooth manifold $M_{\mathrm{sp}}^{n+3}$, a constant $c_n>0$, and metrics $g_j$ such that
	\begin{equation*}
		\begin{gathered}
		\sec(g_j)\geq-1,\qquad \Ric(g_j)\geq c_n g_j,\\
		(M_{\mathrm{sp}}^{n+3},g_j)\xrightarrow{\mathrm{GH}}X_{\mathrm{sp}}^n.
		\end{gathered}
	\end{equation*}
	Here $X_{\mathrm{sp}}^n$ is a compact Alexandrov space whose underlying space is a closed simply connected topological manifold satisfying
	\begin{equation*}
		w_2(X_{\mathrm{sp}}^n)=0,\qquad \operatorname{ks}(X_{\mathrm{sp}}^n)\ne0.
	\end{equation*}
	The class $w_2(X_{\mathrm{sp}}^n)$ is that of its stable topological tangent bundle.
\end{proposition}

\begin{proof}
	Put $k=n-4\geq2$ and retain $L$, $u=c_1(L)$, and $\pi:M_+^7\to\mathcal Q_+^4$ from the preceding construction. Let $L_{\mathbb R}$ be the underlying oriented real plane bundle of $L$, and write $S(E)$ for the sphere bundle of a real vector bundle $E$. Define
	\begin{equation*}
		\begin{gathered}
		E_k=\underline{\mathbb R}^{\,k-1}\oplus L_{\mathbb R},\\
		X_{\mathrm{sp}}^n=S(E_k),\qquad M_{\mathrm{sp}}^{n+3}=S(\pi^*E_k).
		\end{gathered}
	\end{equation*}
	Write $p:X_{\mathrm{sp}}^n\to Q_+^4$ and $a:M_{\mathrm{sp}}^{n+3}\to M_+^7$ for the sphere-bundle projections. Since $L$ is Cartier, $E_k$ is an ordinary vector bundle on the coarse space and has trivial local isotropy action at $v$. Thus $\pi^*E_k$ is a smooth vector bundle, and the $G$-action lifts to a locally free action on $M_{\mathrm{sp}}^{n+3}$ with orbifold quotient $S(E_k)$. A neighborhood of the singular fibre $\Sigma^k=p^{-1}(v)$ is $C(H^3)\times S^k$. Hence \eqref{eq:local-double-suspension} makes $X_{\mathrm{sp}}^n$ a topological manifold. Both sphere bundles have simply connected base and fibre, so both total spaces are simply connected.

	The bundle $E_k$ has a nowhere-zero section, so its Euler class vanishes. The mod-$2$ Gysin sequence therefore makes fibre integration $H^k(X_{\mathrm{sp}}^n;\mathbb Z/2)\to H^0(Q_+^4;\mathbb Z/2)$ surjective. In particular, $[\Sigma^k]\ne0$ in $H_k(X_{\mathrm{sp}}^n;\mathbb Z/2)$. The complement of a cone neighborhood is smooth, and \autoref{lem:link-obstruction} gives
	\begin{equation*}
		\operatorname{ks}(X_{\mathrm{sp}}^n)\cap[X_{\mathrm{sp}}^n]=[\Sigma^k]\ne0.
	\end{equation*}

	Set $\bar u=u\bmod2$ and $X_{\mathrm{reg}}^n=X_{\mathrm{sp}}^n\setminus\Sigma^k$. On the regular part, the stable tangent formula for a sphere bundle gives
	\begin{equation*}
		\begin{aligned}
		w_2(X_{\mathrm{reg}}^n)
		&=p^*\bigl(w_2(T(Q_+^4\setminus\{v\}))+w_2(E_k)\bigr)\\
		&=p^*(\bar u+\bar u)=0.
		\end{aligned}
	\end{equation*}
	Here we used $c_1(T\mathcal Q_+^4)=u$. Excision and the local model $C(H^3)\times S^k$ show that $H^i(X_{\mathrm{sp}}^n,X_{\mathrm{reg}}^n;\mathbb Z/2)=0$ for $i=2,3$, since $H^3$ is a homology sphere. Thus restriction is an isomorphism in degree two, and $w_2(X_{\mathrm{sp}}^n)=0$ globally. For the source, the vertical tangent bundle of $M_+^7\to\mathcal Q_+^4$ is trivial, while the horizontal bundle is the pullback of the orbifold tangent bundle. Consequently $w_2(M_+^7)=\pi^*\bar u$, and the sphere-bundle tangent formula gives
	\begin{equation*}
		w_2(M_{\mathrm{sp}}^{n+3})
		=a^*\bigl(w_2(M_+^7)+w_2(\pi^*E_k)\bigr)=0.
	\end{equation*}
	The source is orientable and hence spin.

	Equip $E_k$ with a Euclidean metric and a compatible orbifold connection. Give $S(E_k)$ the connection metric with horizontal metric $q$ from \autoref{lem:positive-ricci-base} and round fibres of radius $\varepsilon$. Since $k\geq2$, this metric has positive Ricci curvature for all sufficiently small $\varepsilon>0$ by \cite[Theorem 3.4]{BoucettaEssoufi2020}. The local estimates in \cite[Proposition 2.4]{BoucettaEssoufi2020} apply in finite orbifold charts, and compactness makes the choice of $\varepsilon$ uniform. Fix such a radius and apply \autoref{prop:collapse-metrics} to the lifted $G$-action on $M_{\mathrm{sp}}^{n+3}$. Shrinking the $G$-orbits gives the asserted collapse, and one constant rescaling gives $\sec\geq-1$ with a common positive Ricci lower bound.
\end{proof}

The homotopy equivalence $f:Q_+^4\to\mathbb{CP}^2$ in \autoref{lem:fano-topology} satisfies $L\cong f^*\mathcal O_{\mathbb{CP}^2}(1)$ as topological line bundles. The induced map of sphere bundles is a homotopy equivalence, so $X_{\mathrm{sp}}^n$ has the homotopy type of
\begin{equation*}
	S\!\left(\underline{\mathbb R}^{\,n-5}\oplus\mathcal O_{\mathbb{CP}^2}(1)_{\mathbb R}\right).
\end{equation*}

\section{Aspherical collapsing limits}\label{sec:aspherical}

Cap the negative $E_8$ plumbing itself to obtain the closed homology four-manifold
\begin{equation*}
	Q_E^4=W_E^4\cup_{H^3}C(H^3).
\end{equation*}
It has intersection form $-E_8$ and one cone point with link $H^3$. Reverse its orientation so that the input to hyperbolization has signature $8$, as in \cite{DJ}. The specific Davis--Januszkiewicz hyperbolization in \cite[Lemma 5a.3]{DJ} produces a compact polyhedral homology four-manifold $Q_h^4$ with the following properties: it is aspherical and admits a locally $\operatorname{CAT}(0)$ metric; it is a PL manifold away from a single vertex $v_h$ with link the underlying Poincar\'e sphere $H^3$; and
\begin{equation}\label{eq:hyperbolization-invariants}
	w_1(Q_h^4)=w_2(Q_h^4)=0,\qquad\sigma(Q_h^4)=8.
\end{equation}
We use these properties of the particular construction, including its single-vertex local model.

\begin{lemma}\label{lem:hyperbolized-orbifold}
	The underlying space $Q_h^4$ admits a smooth effective orbifold structure $\mathcal Q_h^4$ whose only singular chart is $I^*\backslash D^4$.
\end{lemma}

\begin{proof}
	Remove the interior of a closed PL cone neighborhood of $v_h$. Its complement $W_h^4$ is a compact PL four-manifold with boundary PL-homeomorphic to $H^3$. It admits a compatible smoothing, and its induced boundary smoothing is diffeomorphic to the standard spherical manifold $I^*\backslash S^3$. These low-dimensional smoothing facts are recalled in \cite[Introduction and Section 2.2]{Lange}.

	Choose such a boundary diffeomorphism and glue the standard orbifold ball $I^*\backslash D^4$ to $W_h^4$ using smooth collars. On the overlap both smooth structures are those of an ordinary annulus $H^3\times(0,1)$, so this defines a smooth orbifold. Its underlying space is homeomorphic to $Q_h^4$: the homeomorphism comparing the new boundary identification with the original PL identification extends radially across $C(H^3)$. Together with the identity on $W_h^4$, it gives the required homeomorphism.
\end{proof}

The lower-curvature metric below is chosen on $\mathcal Q_h^4$ after this construction. It is independent of the piecewise-flat locally $\operatorname{CAT}(0)$ metric used to establish asphericity.

\begin{proof}[Proof of \autoref{thm:aspherical}]
	Apply \autoref{lem:smooth-total-space} to $\mathcal Q_h^4$, obtaining a closed smooth manifold $M_h^7$ with locally free $G$-action. Choose any smooth metric $q_h$ on the compact orbifold $\mathcal Q_h^4$. By \autoref{prop:collapse-metrics}, there are metrics on $M_h^7$ with a common sectional lower bound and diameter upper bound converging to $(Q_h^4,q_h)$.

	For $k=n-4\geq1$, take a fixed flat torus $T^k$ and put
	\begin{equation*}
		M^{n+3}=M_h^7\times T^k,\qquad X^n=Q_h^4\times T^k.
	\end{equation*}
	The product metrics have the required convergence and sectional curvature and diameter bounds after one rescaling. By \eqref{eq:scalar-collapse}, their scalar curvatures tend uniformly to $+\infty$; taking a product with a flat torus and applying a fixed rescaling preserve this conclusion. By \eqref{eq:local-double-suspension}, $X^n$ is a closed topological manifold. It is orientable, and a neighborhood of its singular stratum is $C(H^3)\times T^k$, with smooth complement. Hence
	\begin{equation*}
		\operatorname{ks}(X^n)\cap[X^n]=[\{v_h\}\times T^k]\ne0
	\end{equation*}
	by \autoref{lem:link-obstruction} and projection to the torus. Asphericity follows from that of both factors. The product of the original locally $\operatorname{CAT}(0)$ metric on $Q_h^4$ with the flat torus metric gives the separate locally $\operatorname{CAT}(0)$ metric asserted in the theorem.

	Replacing the cone neighborhood in $Q_h^4$ by a contractible topological four-manifold gives the closed aspherical manifold $M_{\mathrm{top}}^4$ of \cite[Theorem 5a.1]{DJ}. The same construction gives a homotopy equivalence $M_{\mathrm{top}}^4\simeq Q_h^4$. Consequently
	\begin{equation*}
		X^n=Q_h^4\times T^k\simeq M_{\mathrm{top}}^4\times T^k.
	\end{equation*}
	By \cite[Corollary 5a.4]{DJ}, this homotopy type has no closed PL representative for any $k\geq1$. The Ricci obstruction follows from \autoref{prop:aspherical-ricci-obstruction} below.
\end{proof}

The final assertion of \autoref{thm:aspherical} follows from the next application of the generalized Margulis lemma and the fundamental-group surjection theorem for Gromov--Hausdorff limits.

\begin{proposition}\label{prop:aspherical-ricci-obstruction}
	Let $(X^n,d)$ be a compact length space whose underlying space is a closed connected aspherical topological manifold with no closed PL representative of its homotopy type. For every fixed integer $N\geq2$, there is no sequence of closed Riemannian manifolds $(M_j^N,g_j)$ such that
	\begin{equation*}
		\begin{gathered}
		(M_j^N,g_j)\xrightarrow{\mathrm{GH}}(X^n,d),\\
		\Ric(g_j)\geq-\varepsilon_jg_j,\qquad
		\varepsilon_j\geq0,\quad\varepsilon_j\longrightarrow0.
		\end{gathered}
	\end{equation*}
\end{proposition}

\begin{proof}
	Suppose that such a sequence exists. Its diameters are bounded because the limit is compact. Put $a_j=\varepsilon_j+j^{-1}$ and $\widetilde g_j=a_jg_j$. Then
	\begin{equation*}
		\Ric(\widetilde g_j)>-\widetilde g_j,\qquad
		\operatorname{diam}(M_j^N,\widetilde g_j)\longrightarrow0.
	\end{equation*}
	The generalized Margulis lemma \cite[Corollary 2]{KapovitchWilking2011} implies that $\pi_1(M_j^N)$ is virtually nilpotent for all sufficiently large $j$. Since $X^n$ is a compact topological manifold, \cite[Theorem 2.1]{SormaniWei2001} gives an epimorphism $\pi_1(M_j^N)\twoheadrightarrow\pi_1(X^n)$ for all sufficiently large $j$. Hence $\pi_1(X^n)$ is virtually nilpotent as well.

	The fundamental group of a closed aspherical manifold is finitely generated and torsion-free. The realization theorem of Lee--Raymond \cite{LeeRaymond1985}, as used in \cite[proof of Corollary 7.1]{KapovitchWilking2011}, identifies it with the fundamental group of a closed smooth infranilmanifold. Both spaces are $K(\pi,1)$ spaces, so they are homotopy equivalent; the infranilmanifold has dimension $n$, as follows also from cohomological dimension. This contradicts the assumed absence of a closed PL representative of the homotopy type of $X^n$.
\end{proof}

\begin{remark}\label{rem:diameter-normalization}
	In \autoref{thm:aspherical}, the diameter bound is a finite constant supplied by the construction. The proof does not assert the existence of these aspherical examples for every prescribed positive diameter bound while keeping the sectional curvature normalization $\sec\geq-1$.
\end{remark}

\begin{center}
\footnotesize{ACKNOWLEDGMENTS}
\end{center}
The second author is very grateful to Professor Burkhard Wilking for his enlightening lectures on advanced differential geometry and for suggesting the study of Cheeger deformations and collapsing theory. He also sincerely thanks Professor Wilking for his warm hospitality and generous help during his stay in M\"{u}nster.

\section*{AI Assistance Disclosure}
ChatGPT (OpenAI) was used during the preparation of this manuscript
to assist with English-language writing and editing, as well as
the organization of relevant literature. The authors take full
responsibility for the correctness and content of the manuscript.

\begingroup
\setlength{\emergencystretch}{3em}
\printbibliography
\endgroup
\end{document}